\documentclass{amsart}

\usepackage{amssymb}
\usepackage{graphicx}
\usepackage{color}

\newtheorem{theorem}{Theorem}[section]
\newtheorem{lemma}[theorem]{Lemma}
\newtheorem{remark}[theorem]{Remark}

\newcommand{\jump}[1]{[\hspace{-2.0pt}[ #1 ]\hspace{-2.0pt}]}
\newcommand{\mean}[1]{\{\hspace{-3.5pt}\{ #1 \}\hspace{-3.5pt}\}}
\newcommand{\nE}{n}

\newcommand{\beq}{\begin{equation}}
\newcommand{\eeq}{\end{equation}}

\def\ds{\displaystyle}
\newcommand{\nn}{\nonumber}

\newcommand{\cT}{\mathcal{T}}
\newcommand{\cE}{\mathcal{E}}
\newcommand{\cV}{\mathcal{V}}
\newcommand{\hu}{\hat u}

\newcommand{\dT}{\partial T}

\newcommand{\setR}{\mathbb{R}}
\newcommand{\al}{\alpha}
\newcommand{\si}{\sigma}
\newcommand{\Ga}{\Gamma}
\newcommand{\Om}{\Omega}

\newcommand{\opdiv}{\operatorname{\rm div}}

\newcommand{\dn}{\partial_n}

\newcommand{\Vnull}{V_h^0}

\begin{document}

\title[An Equilibration Based Error Estimate
   for the Biharmonic Equation]{An Equilibration Based A Posteriori Error Estimate
   for the Biharmonic Equation and Two Finite Element Methods }

\author[D.\ Braess]{Dietrich  Braess}
\address{Dietrich Braess\\ Faculty of Mathematics\\ Ruhr-University\\ D-44780 Bochum, Germany}
\email{dietrich.braess@rub.de}
\author[A.S.\ Pechstein]{Astrid S.\ Pechstein}
\address{Astrid. S.\ Pechstein\\ Institute of Technical Mechanics\\ Johannes Kepler University Linz\\ Altenbergerstr. 69\\ 4040 Linz, Austria}
\email{astrid.pechstein@jku.at}
\author[J.\ Sch\"oberl]{Joachim Sch\"oberl}
\address{Joachim\ Sch\"oberl\\ Institute for Analysis and Scientific Computing\\ Vienna University of Technology\\ Wiedner Hauptstrasse 8-10\\ 1040 Wien, Austria}
\email{joachim.schoeberl@tuwien.ac.at}

\begin{abstract}
We develop an a posteriori error estimator for the Interior Penalty 
Discontinuous Galerkin approximation of the biharmonic equation with 
continuous finite elements.
The error bound is based on the 
two-energies principle and requires the computation of an equilibrated
moment tensor.    
The natural space for the moment tensor
consists of symmetric tensor fields
with continuous  normal-normal components. It
is known from the 
Hellan-Herrmann-Johnson (HHJ) mixed formulation. We propose a
construction that is totally local.
The procedure can also be applied to the original HHJ formulation,
which directly provides an equilibrated moment tensor. 
\end{abstract}

\subjclass[2010]{Primary hprimary classi;
Secondary hsecondary classesi}


\thanks{The authors want to thank Carsten Carstensen for providing us with
the numerical constant in the estimate of the data oscillation.}
\maketitle



\section{Introduction}\label{sec:Intro}
\setcounter{equation}{0}
The numerical solution of the biharmonic equation by the discontinuous
Galerkin method attracts interest in order to avoid $H^2$-conforming
elements.
The classical formulation of the biharmonic equation reads: {\em find $u \in H^2_0(\Om)$ such that}
\beq
\label{strong}
  \Delta^2 u = f.
\eeq

In the framework of plate theory, the biharmonic equation is used as a model for  Kirchhoff plates.
The present paper refers to the Hellan--Herrmann--Johnson plate formulation \cite{HEL67, HER67, JOH73}
with two equations of second order,
\begin{align}
  \label{HHJ}
  \left. \begin{array}{r} \nabla^2 u = \si , \\
  \opdiv\opdiv \si = f . \end{array} \right.
\end{align}

In the context of plate theory, the scalar function $u$ represents the deflection and the tensor field $\si$  the bending moment.
For
generalizations and error estimates of the Hellan--Herrmann--Johnson formulation see
\cite{AB85,BdVNS08,COM80,GHV11}.

The DG methods for the treatment of \eqref{HHJ} depart from the weak formulation:
{\em find $u \in H^2_0(\Om)$ such that}
\beq
\label{weak}
  \int_{\Om} \nabla^2 u : \nabla^2 w \,dx = \int_{\Om} f w\,dx
  \quad\text{ for all~} w\in H^2_0(\Om).
\eeq
Penalty terms are added to the corresponding energy functional
in order to deal with the nonconforming elements;
see the early work for fully discontinuous elements \cite{Baker:77}.
The case of continuous, but not continuously differentiable $C^0$ elements
is treated in Section~\ref{C0IPDG} below.

Several a posteriori estimates of residual type can be found
in the literature \cite{BON10, BGS10, FHP15, GHV11, VER13}. 
Recently an a posteriori error estimate has been established by the
two-energies principle (hypercircle method) for the full discontinuous interior
penalty (IPDG) method \cite{BHL15}, where the finite elements for the $u$-variable
are not even $H^1$-conforming.

In this paper we turn to the interior penalty discontinuous Galerkin method
with continuous finite elements ($C^0$IPDG). Here only jumps in the derivatives need
to be penalized. Although the difference to the above mentioned IPDG method seems to be small,
the two-energies principle requires a quite different approach here.

The main part of the discretization error will be evaluated by use of a
tensor $\sigma_h^{eq}$ of bending moments with the  equilibration property
\beq
\label{divdiv-si}
  \opdiv\opdiv\sigma_h^{eq}=f_h.
\eeq
We will consider the operator $\opdiv\opdiv$ as a  differential operator in distributional sense.
It has been  analyzed in the framework of the Tangential Displacement Normal Normal Stress method \cite{PS11,PS16} for continuum mechanics. 
The right-hand side $f_h$ is a finite element
approximation of $f$ in the distributional sense.
The tensor $\sigma_h^{eq}$ is taken from the space of Hellan--Herrmann--Johnson elements
which are symmetric, piecewise polynomial tensors with continuous normal-normal components.

The equilibrated tensor $\sigma_h^{eq}$ will be computed by a postprocessing which uses only  local procedures.
The analysis for the nonconforming 
DG method is more involved
than for the mixed method with Hellan--Herrmann--Johnson elements
although there is a great similarity. 
It shows that the DG method may be considered as a formulation
between a primal and a mixed method.

The present paper is organized as follows:
Section 2 lists some notation.
In Section 3 we introduce the two-energies principle for the biharmonic
equation with the distributional form of the double divergence operator.
Moreover we discuss the treatment of nonconforming (i.e.\ non-$C^1$) elements.
Section 4 presents the $C^0$IPDG version of the discontinuous
Galerkin method.
Section 5 is devoted to the equilibration procedure, and Section 6
deals with the data oscillation.
The efficiency of the resulting a posteriori error bound is shown in Section 7. 
A short excursion to the Hellan--Herrmann--Johnson element
and a corresponding a posteriori error estimate follows in Section 8.
Numerical results in Section 9 verify the theoretical results and show
how other boundary conditions are covered.

\section {Notation}
%
%
We consider
the biharmonic equation on a bounded, open polygonal Lipschitz domain $ \Om \subset \mathbb{R}^2$.
Let $ \cT_h$ be a geometrically conforming,
locally quasi-uniform simplicial triangulation of $\Om$.
We denote the sets of edges and of vertices by $ \cE_h$ and $\cV_h$ including boundary edges and vertices, respectively. We write $\cE_h^0$ and $\cV_h^0$ for the subsets contained in the interior of $\Omega$.
Given
an edge or element $D \in \cT_h \cup \cE_h$ and $m \in \mathbb{N}$, we refer to $ P^{m}(D)$ as the set of polynomials of degree $ \le m$ on $D$. 
The set of symmetric $2\times2$ tensors with components in $P^m(D)$ is referred to as $[P^m(D)]^{2\times2}_{sym}$.

We denote the outward unit normal vector of an element $T \in \cT_h$ by $n$
and obtain the tangential vector $t$ by rotating $n$ by $\pi/2$.
We consider all edges as oriented, i.e., an edge is pointing from vertex $V_1(E)$
to vertex $V_2(E)$. We refer to $T_1(E)$ as the element on the left-hand side of $E$, while $T_2(E)$ lies on the right-hand side;
only $T_1(E)$ exists for edges on the boundary. The normal and tangential vector of an edge $E$ shall coincide with those of $T_1(E)$.

A piecewise continuous tensor field $\tau$ on $\Om$ has a normal vector $\tau_{n} = \tau n $ on the boundary of each element $T$. The normal vector can be decomposed into a (scalar) normal and tangential component, $\tau_{nn} = \tau_{n} \cdot n$ and $\tau_{nt} = \tau_{n} \cdot t$. Note that $\tau_{nn}$ and $\tau_{nt}$ are invariant under a change of orientation of $n$ and $t$.

Let $E$ be an interior edge shared by elements $T_1 = T_1(E)$ and $T_2 = T_2(E)$.
Given a scalar function with smooth restrictions
$\phi_i :=  \phi|_{T_i}$, we define the average and the jump 
\[
  \mean{\phi} := \frac12 (\phi_1+\phi_2), \quad
   \jump{\phi} := \phi_1 - \phi_2
  \quad\text{on } E \in \cE_h^0.
\]
This definition holds also for $\phi$ being a scalar-valued tensor component.
We further need the jump of the normal derivative,
\[
   \jump{\partial_n \phi} := \jump{\nabla\phi}\cdot n
 =\nabla\phi_1\cdot n_1 + \nabla\phi_2\cdot n_2
  \quad\text{on } E \in \cE_h^0.
\]
Although the jump $\jump{\phi}$ does depend on the orientation of the edge, it will only occur in products with other quantities that depend on the orientation. The final outcome is then invariant.
Jump and average are defined on a boundary edge $E \subset \Gamma$ by
\[
  \mean{\phi} := \phi_1, \quad
   \jump{\phi} := \phi_1, \quad \jump{\partial_n \phi} = \nabla\phi_1\cdot n_1.
  \quad\text{on } E \in \cE_h \backslash\cE_h^0.
\]



We will use standard notation from Lebesgue and Sobolev space
theory. We denote the $L_2$-inner product and the associated $L_2$-norm of $\Omega$ 
 by $ (\cdot,\cdot)_{0,\Om} $ and $ \| \cdot \|_{0,\Om} $, respectively.
The product $\langle \cdot, \cdot \rangle$ denotes a duality pairing.

Finite element spaces will be involved 
that are only piecewise $H^2$ function on $\cT_h$. 
The double gradient is understood as a pointwise derivative denoted by $\nabla_h^2$,
e.g., in the broken seminorm
\begin{equation}
  |v|_{2,h}^2 := \|\nabla_h^2 v\|_{0,\Omega}^2 = \sum_{T\in \cT_h} \|\nabla^2 v\|_{0,T}^2 \,.
	\label{eq:brokenH2}
\end{equation}
\section{A two-energies principle for the biharmonic equation}
\label{sec:TwoEnergies}
\subsection{The principle}
The two-energies principle was originally established by Pra\-ger and Synge
\cite{PRS47, SYN47} for elliptic equations of second order
under the name {\em hypercircle method}.
It has been used by many authors,
e.g., in \cite{AIR10,BRA07,BFH14, BHS08, ENV07, Stenberg}
for the evaluation of a posteriori error estimates.
The principle was reformulated several times in order to obtain
error estimates by a postprocessing also when nonconforming
finite elements are involved.

The principle was formulated for problems of fourth order in \cite{NR01}
and used for computing a posteriori error bonds in \cite{BHL15}.
It is based on the fact that there is no duality gap between the minimum problem
\beq
\label{maxprob}
  \frac12 \int_{\Om} (\nabla^2 w)^2dx -\int _{\Om} f w\,dx 
  \longrightarrow \min_{w \in H^2_0(\Om)} !
\eeq
and the complementary maximum problem
\begin{eqnarray}
\label{minprob}
  &&-\frac12 \int_{\Om} \tau^2dx \longrightarrow
   {\max_{\tau \in L_2(\Om)^{2\times2}_{sym} }!} \hspace{1cm}  \\[3pt]
  \text{subject to}\hspace{-0.8cm}&&\qquad \opdiv\opdiv \tau = f. \nn
\end{eqnarray}
Nevertheless, the application to elliptic problems of order four
requires special actions.

Here and throughout the paper we will apply the principle with the
differential operator $\opdiv\opdiv$ in distributional form,
\beq
\label{divdivdistr}
  \langle \opdiv\opdiv \tau, w \rangle := \int_{\Om}\tau :\nabla^2 w\,dx, \quad
  \tau \in [L_2(\Om)]^{2\times2}_{sym}, ~ w \in H^2_0(\Om).
\eeq
Although the right-hand side of the original equation \eqref{HHJ}
is assumed to be in $L_2(\Om)$,
it is essential that we have a distributional version of the principle in $H^{-2}$.
Then we can choose tensor from the Hellan--Herrmann--Johnson space
as an equilibrated moment tensor.
Obviously \eqref{maxprob} is well defined also for $f \in H^{-2}$.

\begin{theorem}
\label{theo-2en}
(Two-energies principle for the biharmonic equation)  \\  
Let  $f_h \in H^{-2}(\Om)$ and $\hu\in H^2_0(\Om)$ be the solution of the biharmonic equation
\beq
\label{weak-eq}
  \int_{\Om} \nabla^2 \hu : \nabla^2 w\, dx =\langle f_h, w\rangle
  \quad\text{for all } w\in H^2_0(\Om).
\eeq
If  $v\in H^2_0(\Om)$ and the tensor
$\sigma^{eq}_h \in [L_2(\Om)]^{2\times 2}_{sym}$ is equilibrated
in the sense that
\beq
\label{equil}
  \left< \opdiv\opdiv\sigma^{eq}_h,w \right> = \langle f_h, w\rangle
  \quad\text{for all } w \in H^2_0(\Om)
\eeq
then
\beq
\label{2-en-for}
  \int_{\Om} (\nabla^2(\hu-v))^2dx +  \int_{\Om} (\nabla^2
   \hu-\sigma^{eq}_h)^2dx
  =  \int_{\Om} (\nabla^2 v-\sigma^{eq}_h)^2dx.
\eeq
\end{theorem}

\begin{proof}
By the definition of the distribution and by the equilibration we have
\beq
\label{all-obc}
  \int_{\Om} \sigma^{eq}_h : \nabla^2 w\, dx
  = \left< \opdiv\opdiv\sigma^{eq}_h , w\right> = \langle f_h,w\rangle
   \quad\text{for all }   w \in H^2_0(\Om).
\eeq
Combining this equation with \eqref{weak-eq} we obtain with $w:=\hu-v$:
\begin{eqnarray*}
  \lefteqn {\int_{\Om} (\nabla^2 \hu-\sigma^{eq}_h): \nabla^2(\hu-v)dx }  \\
  &\quad =& \int_{\Om} \nabla^2 \hu : \nabla^2(\hu-v)dx 
   - \int_{\Om} \sigma^{eq}_h: \nabla^2(\hu-v)dx \\[4pt]
  &\quad =&\langle f_h, \hu-v\rangle -\langle f_h, \hu-v\rangle =0.
\end{eqnarray*}
This orthogonality relation and the Binomial formula yield \eqref{2-en-for}.
\end{proof}

The generalization of Theorem \ref{theo-2en} to other boundary conditions
will be described in Remark \ref{obc}.

\subsection{Error estimation using the two-energies principle}

The dominating part of the overall discretization error
will be estimated by using the two-energies principle~\eqref{2-en-for}. To this end, an equilibrated moment tensor
$\sigma_h^{eq}$
will be constructed.
As was pointed out in \cite{BHL15}, we usually get two additional terms
in a posteriori error estimates.

The finite element solution $u_h$ of
the $C^0$IPDG method is contained only in $H^1(\Om)$.
We need an $H^2$ function $v$ in order to apply Theorem \ref{theo-2en}.
An interpolation by a Hsieh--Clough--Tocher element,
by an element of the TUBA family \cite{AFS68} or by another $H^2$-function $u^{conf}$ implies an additional term $|u_h-u^{conf}|_{2,h}$.
This  term does not spoil the efficiency, since it can be bounded by terms
of residual a posteriori error estimates that are known to be 
efficient \cite{BHL15,BGS10}.

Another extra term is induced by the so-called data oscillation. For general $f \in L_2(\Omega)$, the discrete equilibrated moment tensor $\sigma^{eq}_h$ is not equilibrated with respect to $f$,
\begin{equation}
\langle \opdiv\opdiv \sigma_h^{eq}, w\rangle \not= (f, w)_{0,\Omega} \qquad \forall v \in H^2_0,
\end{equation}
but
\begin{equation}
  \langle \opdiv\opdiv \sigma_h^{eq}, w\rangle = \langle f_h, w\rangle \qquad \forall v \in H^2_0.
\end{equation}
The choice of $f_h$ will be explained in Section \ref{sec:dataosc}; so far we only mention that $f_h$ can be seen as the interpolation of $f$ to a discrete distributional space. The difference between $f \in L_2$ and $f_h = \opdiv\opdiv\sigma^{eq}_h$ constitutes the last term in the sum \eqref{a-post} below.

To be specific, let $u \in H^2_0$ denote the solution
of the given biharmonic equation, and $u_h$ be the discrete solution obtained by a DG method.
Since $u_h \notin H^2_0(\Omega)$, we estimate the error $u - u_h$ in the broken $H^2$ norm \eqref{eq:brokenH2} or the mesh-dependent DG norm \eqref{IPDGNorm1}
below, which includes jumps of the normal derivative across edges.
Inserting the interpolant of $u_h$ to an $H^2$-conforming finite element space $u_h^{conf}$ and the solution $\hat u \in H^2$ to the biharmonic equation with modified right hand side $f_h \in H^{-2}$, we obtain the following error estimate by the triangle inequality,
\begin{eqnarray} 
  \lefteqn{|u_h-u|_{2,h}  } \nn \\ 
  & \leq&
	|u_h - u^{conf}|_{2,h} + |u^{conf} - \hat u|_2 + |\hat u - u|_2 \nn\\
	& \leq& \underbrace{|u_h - u^{conf}|_{2,h}}_{\eta^{nonconf}}
   +\underbrace{\|\nabla^2 u^{conf}-\si_h^{eq}\|_{0,\Om}}_{\eta^{eq}}
  +\underbrace{ \|\opdiv\opdiv\sigma_h^{eq} - f\|_{-2}}_{\eta^{osc}}\,. \quad
 \label{a-post}
  \label{eq:errorest}
\end{eqnarray}
The  term $\eta^{eq}$ on the right-hand side of \eqref{eq:errorest} is obtained by the two-energies principle, and it
is the dominating one.
The  term $\eta^{osc}$ stems from the data oscillation as treated in Section~\ref{sec:dataosc}.
There it will be shown that the order is at least $c h^2 $.

\subsection{An improvement for nonconforming elements}
The estimate \eqref{eq:errorest} can be improved 
for nonconforming methods by a
simple consideration \cite{Stenberg}.
It is now appropriate to recall
the name hypercircle method given by Prager und Synge \cite{SYN47}.
The computation incorporates the center of the hypercircle, i.e.,
the mean value
$\sigma^{mean} := 1/2(\nabla^2 u^{conf} + \sigma^{eq})$.
The orthogonality of two sides of the triangle
in the hypercircle implies
\begin{equation}
\begin{split}
  &\|\nabla^2 \hat u -  \sigma^{mean} \|_{0,\Om}^2 =\\
 &=  \|\tfrac12(\nabla^2 \hat u -  \sigma^{eq} )
   + \tfrac12(\nabla^2 \hat u - \nabla^2 u^{conf}) \|_{0,\Om}^2 \\
 &=  \|\tfrac12(\nabla^2 \hat u -  \sigma^{eq} )
   - \tfrac12\nabla^2 (\hat u - u^{conf}) \|_{0,\Om}^2
   +\underbrace{(\nabla^2 \hat u -  \sigma^{eq}, 
   \nabla^2 (\hat u - u^{conf}))_{0,\Om}}_{=0}
   \\
 & = \|\tfrac12 ( \nabla^2 u^{conf}-\sigma^{eq})\|_{0,\Om}^2
   \,.
\end{split}
\label{meansig}
\end{equation}
Now the auxiliary point in the triangle equality \eqref{eq:errorest} will be $\sigma^{mean}$
instead of $\nabla^2 u^{conf}$, and \eqref{meansig} is used. 
We obtain the improved error estimate
\begin{equation}
\label{eq:errorest_imp}
\begin{split}
 &|u_h - u|_{2,h}  \le \\&\leq  \|\nabla_h^2 u_h - \sigma^{mean} \|_{0,\Om}
   + \|\sigma^{mean} - \nabla^2 \hat u\|_{0,\Om} +  |\hat u - u|_2 \\
	&\leq \ \underbrace{ \|\nabla_h^2 u_h - \sigma^{mean} \|_{0,\Om}
     }_{\eta^{mean}} + 
		\frac12 \underbrace{  \|\nabla^2 u^{conf} - \sigma^{eq}\|_{0,\Om}}
     _{\eta^{eq}} + \underbrace{ \|\opdiv\opdiv\sigma_h^{eq} - f\|_{-2}}_{\eta^{osc}}.
\end{split}
\end{equation} 
The sketch on the right hand side of Figure~\ref{fig:ee} 
and the triangle inequality $\eta^{mean} \le \eta^{nonconf}+\frac12 \eta^{eq}$
ensure that the 
estimate \eqref{eq:errorest_imp} is at least as good as the original one \eqref{eq:errorest}.

\begin{figure}
\begin{center}
\includegraphics[width=0.4\textwidth]{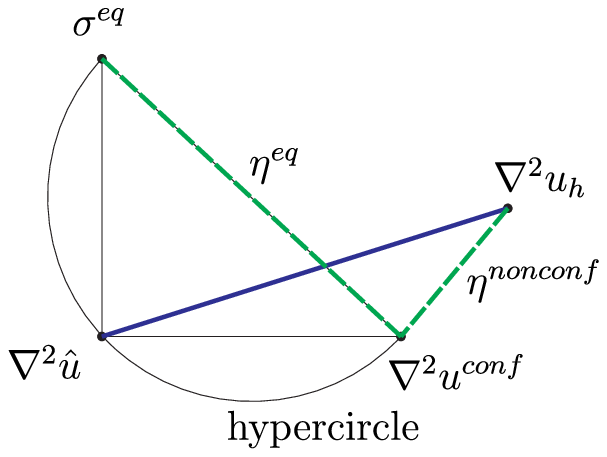}
\hspace{0.05\textwidth}
\includegraphics[width=0.4\textwidth]{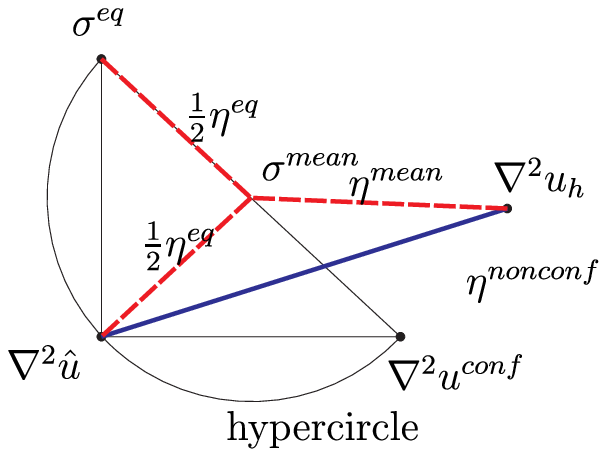}
\end{center}
\vspace{-0.2cm}
\caption{Error estimation using the hypercircle method. \break
Left: original estimate, right: improved estimate}
\label{fig:ee}
\end{figure}

\section {Discretization of the biharmonic equation}
\label {C0DG}
\subsection{The $C^0$IPDG method} \label{C0IPDG}

%
A popular way for the numerical treatment of the biharmonic equation is the
interior penalty ($C^0$IPDG) method; see, e.g., \cite{BGS10, FHP15, SUM07}.
We assume that $f\in L_2(\Om)$.
Given $k\ge 2$, the DG method uses the polynomial finite element spaces
\beq
  \label{DGSpace1}
    V_h := \{ v_h \in C^0(\Om) \ | \ v_h|_T\in P_k(T) , \ T \in  \cT_h \}
\eeq
and $\Vnull := V_h \cap H^1_0(\Om)$.

The DG bilinear form
$A_h(\cdot,\cdot) :  V_h \times V_h \rightarrow \setR$ contains a penalty term
with a sufficiently large penalty parameter $\al$,
\begin{align}
\label{IPDGBilForm1}
   A_h(u_h,v_h) & := \sum\limits_{ T \in \cT_h} 
    \int_ T  \nabla^2 u_h : \nabla^2 v_h \,dx     \\
    \nn
  & -\sum\limits_{E \in \cE_h}
 \int\limits_E \Big( \jump{ \dn u_h} \mean{ \nabla^2 v_{h,nn} } \ 
  + \mean{ \nabla^2 u_{h,nn}}\jump{ \dn v_h}  \ \Big) \,ds  \\
 \nn
 &  +\sum\limits_{E \in \cE_h} \int\limits_E \frac{\alpha}{h_E} \ 
   \jump{ \dn u_h } \, \jump{ \dn v_h } \,ds  .
\end{align}
The variational formulation reads.
{\em Find $u_h\in \Vnull$ such that}
\beq
\label{Fe-sol}
  A_h(u_h,v_h) = (f,v_h)_{0,\Om} \quad\text{for all ~} v_h\in \Vnull\,.
\eeq
The discretization error will be measured by the mesh-dependent DG norm on $\Vnull + H_0^2(\Om)$,
\begin{align}
 \label{IPDGNorm1}
   \| v \|^2_{DG}
  := & \ \sum\limits_{T \in \cT_h(\Om)} \| \nabla^2 v \| ^2_{0,T}
  +  \sum\limits_{E \in \cE_h(\bar{\Om})} \frac{\alpha}{h_E} \ \| \, \jump{\dn v} \|^2_{0,E}.
 \end{align}
It is well known that 
there exists a positive constant $\gamma$ such that
\begin{align}
 \label{Ellipticity1}
  A_h(v_h,v_h) \ge \gamma \ \| v_h \|^2_{DG}  \quad v_h \in \Vnull \,,
\end{align}
provided that the penalty parameter $\alpha= O((k+1)^2)$ is sufficiently large.
The bilinear form is also bounded
$
 |A_h(v_h,w_h)| \le c\|v_h\|_{DG} \|w_h\|_{DG}.
$
For the convergence analysis we refer, e.g., to \cite{BGS10, FHP15, SUM07}.

\subsection{The deflection space $V_h$ and its dual}

We will consider degrees of freedom of the deflection space $V_h$ in
\eqref{DGSpace1} that guarantee global continuity of the piecewise polynomials
\begin{subequations}
\begin{eqnarray}
\label{basisduala}
  v_h(x), &~&  x\in\cV_h, \\[3pt]
\label{basisdualb}
  \int_E v_hq\,ds, &&  q\in P^{k-2}(E), ~ E\in \cE_h, \\
\label{basisdualc}
   \int_T v_hq\,dx, &&  q\in P^{k-3}(T), ~ T\in \cT_h. 
\end{eqnarray}
\end{subequations}
These degrees of freedom span the dual space 
\begin{equation}
  V_h^* = \mbox{span}\,(\text{functionals on $V_h$ in 
   \eqref{basisduala}-\eqref{basisdualc}}). \label{eq:defVh*}
\end{equation}
The linear independence may be shown by
proceeding from the vertices to the edges and then to the triangles.
The procedure is elucidated for the analogous three-dimensional case
in the proof of \cite[Lemma 5.47]{MON03}. The degrees of freedom of $\Vnull$ are those functionals in \eqref{basisduala} - \eqref{basisdualc} that are associated with $T \in \cT_h$, $E \in \cE_h^0$ and $V \in \cV_h^0$. These degrees of freedom span the dual space $(\Vnull)^*$. 

An interpolation operator $I_h : H^2(\Om) \to V_h$ is defined for these
degrees of freedom by the conditions; cf. \cite[Proposition 3.2]{COM80},
\begin{eqnarray}
\label{I_h}
  I_h v(x) =v(x) \qquad&~&  x\in\cV_h, \nn\\[3pt]
  \int_E I_h vq\,ds = \int_E  vq\,ds, &&  q\in P^{k-2}(E), ~ E\in \cE_h, \\
   \int_T I_hvq\,dx = \int_T vq\,dx , &&  q\in P^{k-3}(T), ~ T\in \cT_h. \nn
\end{eqnarray}
Obviously, the interpolation operator acts in a local way, and maps $H^2_0(\Omega) \to \Vnull$.
The following local estimate of the interpolation error is well known,
\beq
  \|v - I_h v\|_{0,T} \leq c h_T^{2} \,\|\nabla^2 v\|_{0,T}.	\label{eq:interrorl2}
\eeq


\section {Equilibration}
\label{sec-eq}
The design for determining an equilibrated moment tensor $\sigma_h^{eq} \in M_h$ satisfying
\begin{equation}
\label{equilibratedtau}
  \left< \opdiv\opdiv \si^{eq}_h, v_h\right> = (f,v_h)_{0,\Om}
  \quad   \forall v_h\in \Vnull,
\end{equation}
in distributional sense is 
our first aim and
the main task of this section. The computation will  be done explicitly by a local
postprocessing, but Theorem~\ref{theo-2en} indicates already
that it is done on a different basis than for the IPDG method
in \cite{BHL15}.

We want to find $\sigma^{eq}_h$ such that $\langle \opdiv\opdiv\sigma^{eq}_h, v\rangle$ can be evaluated for less smooth $v \in H^2_0 + \Vnull$.
To this end, we propose to use the finite element space that
is often found in connection with the HHJ method;
see e.g.\ \cite{AB85,COM80,KZ14}.
This space $M_h$ consists of symmetric piecewise polynomial  tensor fields of order $k-1$ with continuous normal-normal component $\tau_{h,nn} = n^T \tau_h n$,
\begin{align}
\label{DGSpace2a}
  M_{h} := \{ \tau_{h} \in  [L_2(\Om)]^{2 \times 2}_{sym} \ |
      & \ \tau_{h}|_T \in [P^{k-1}(T)]^{2 \times 2}_{sym} , \, T\in \cT_h,
     \\
      & \tau_{h,nn} \text{~ is continuous at interelement boundaries}  \} . \nn
\end{align}
Note that the sign of the normal-normal component $\tau_{h,nn}$ does not depend on the orientation of the normal vector.
Comodi \cite[Proposition 3.1]{COM80} presents the following degrees of freedom for the space $M_h$ that take into account the continuity of the normal-normal components
on interelement boundaries.

\begin {lemma}
\label{lem1}
Each $\tau_h \in M_h$ is uniquely defined by the quantities
\beq
  \begin{array}{ll}
  \int_E  \tau_{h,nn}  q_E\,ds,~  & q_E  \in P^{k-1}(E),~ E \in \cE_h, \\[8pt]
   \int_T \tau_h:q_T\, dx,        & q_T \in [P^{k-2}(T)]^{2\times 2}_{sym},~ T\in\cT_h.
\end{array}
\eeq
\end{lemma}


Let $\tau_h \in M_h$ and $w \in H^2_0(\Om)$, then by definition \eqref{divdivdistr}
\begin{equation}
\begin{split}
 \langle \opdiv\opdiv \tau_h, w \rangle
  &=\ \int_\Omega \tau_h : \nabla^2 w\, dx  \\
  &=\ \sum_{T \in \cT_h}  \int_T \tau_h : \nabla^2 v_h\, dx
  { - \sum_{E\in\cE_h} \int_E \tau_{h,nn} \left\jump{\partial_n w\right} \,} ds. 
\end{split}\label{eq:defdivdiv_eps}
\end{equation}
Note that the jump terms $\jump{\partial_n w}$  in \eqref{eq:defdivdiv_eps} vanish for $w \in H^2_0(\Om)$,
but they are relevant for an extension to $H^2_0(\Om)+\Vnull$.

We will use the degrees of freedom \eqref{eq:defdivdiv_eps} for the construction of equilibrated moment tensors.
%
Let $u_h$ be the solution of the finite element equation \eqref{Fe-sol},
i.e., the solution of the $C^0$IPDG method.
By Lemma \ref{lem1} there exists $\si^{eq}_h \in M_h$ such that for
each $T\in \cT_h$, 
\begin{equation}
\label{constr-eq-DG}
\begin{split}
   \si^{eq}_{h,nn} & =  \mean{ \nabla^2 u_{h,nn} } -\tfrac{\alpha}{h}\, \jump{\partial_n   u_h}
         \qquad\in P^{k-1}(E), E\subset \dT,\\[6pt]
  \ds \int_T \si^{eq}_h : q_T\,dx & =  \ds \int_T \nabla^2 u_h : q_T \,dx
  - \sum_{E \subset \partial T} \int_{E} \gamma_E\jump{\partial_n u_h}\, q_{T,nn}\,ds \\ 
	& \hspace{6cm} \forall q_T \in [P^{k-2}(T)]^{2\times 2}_{sym}.
	\end{split}
\end{equation}
In the second line on \eqref{constr-eq-DG}, the factor $\gamma_E$ equals $\gamma_E = 1/2$ for an interior edge $E \in \cE_h^0$, and $\gamma_E = 1$ for a boundary edge $E \subset \Gamma$. 
We insert the equations \eqref{constr-eq-DG} into  \eqref{eq:defdivdiv_eps}
after setting piecewise $q_T:=\nabla^2 v_h$. The choice of $\gamma_E$ ensures that after an edge-wise reordering of boundary integrals in the second line of \eqref{divdiv-f-2} we obtain the integrand $\left\jump{ \partial_n u_h\right}\mean{\nabla^2 v_{h,nn}}$, 
\begin{equation}
\label{divdiv-f-2}
\begin{split}
&\langle \opdiv\opdiv \si^{eq}_h, v_h \rangle
  = \sum_{T \in \cT_h}  \int_{T} \si^{eq}_{h}:\nabla^2 v_h\,dx 
   - \sum_{E \in \cE_h} \int_{E} \si^{eq}_{h,nn} \left\jump{\partial_n v_h\right} \,ds\\
  &= \sum_{T \in \cT_h} \Big( \int_{T} \nabla^2 u_h:\nabla^2 v_h\,dx 
    -  \sum_{E \subset \partial T} \int_{E} \gamma_E\left\jump{ \partial_n u_h\right}
   \nabla^2 v_{h,nn} ds \Big) \quad \\
  &\quad -\sum_{E \in \cE_h} \int_{E} \Big( \mean{\nabla^2u_{h,nn}}
    - \frac{\al}{h} \left\jump{ \partial_n u_h\right} \Big) \left\jump{ \partial_n v_h\right} \, ds  \\
   &= A_h(u_h,v_h)  = (f,v_h)_{0,\Om}.
\end{split}
\end{equation}
The last equality is due to the fact that $u_h$ satisfies the DG equation \eqref{Fe-sol} for all $v_h \in \Vnull$.
It follows from \eqref{divdiv-f-2} that the first aim \eqref{equilibratedtau}
is achieved.

Now we consider the question: for which system is $\si_h^{eq}$
an equilibrated moment tensor? 
For answering this question we set $f_h := \opdiv\opdiv \si_h^{eq}$,
more precisely
\beq
\label{def-f_h}
  \langle f_h, w\rangle =  \langle\opdiv\opdiv \si_h^{eq},w\rangle
  \quad\text{for all } w \in H^2_0(\Om).
\eeq
%
Then $\si_h^{eq}$ is an equilibrated tensor for the
biharmonic equation with the right-hand side $f_h$ by definition.
The next lemma is devoted to a further representation of the double divergence
operator, which indicates that $f_h = \opdiv\opdiv\si_h^{eq}$ lies in the finite-dimensional dual space $(\Vnull)^*$ of the deflection space $\Vnull$. 
This fact will be used to estimate the data oscillation in Section~\ref{sec:dataosc}.

\begin{lemma} 
\label{lem:divdivmwh}
The distributional double divergence operator
$
  \opdiv\opdiv: M_h \longrightarrow (\Vnull)^*
$
which is defined by \eqref{eq:defdivdiv_eps}
is well defined, and there is a representation of the form
\beq
\label{form}
  \langle \opdiv\opdiv \tau_h, v \rangle = \sum_{V \in \cV_h^0} f^{(V)}_\tau v(V)
  + \sum_{E \in \cE_h^0} \int_E f^{(E)}_\tau v\, ds + \sum_{T \in \cT_h} \int_T f^{(T)}_\tau v\, dx
\eeq
with $f^{(V)}_\tau \in \mathbb R$, $f^{(E)}_\tau \in P^{k-2}(E)$
and $f^{(T)}_\tau \in P^{k-3}(T)$.
It contains only evaluations of $v$, but no derivatives of $v$.
The equations \eqref{eq:defdivdiv_divI} and \eqref{ingredients} below
provide equivalent  extensions to all $v \in H^2_0(\Om)+\Vnull$.
\end{lemma}
\begin{proof}
Let $\tau_h \in M_h$ and $v \in H^2_0(\Om)+V_h^0$.
We start from \eqref{eq:defdivdiv_eps},
and partial integration yields
\begin{equation} \label{divdiv1}
\begin{split}
  \left< \opdiv\opdiv \tau_h, v\right> 
  =&\ \sum_{T \in \cT_h} \left( -\int_T \opdiv \tau_h \cdot \nabla v \, dx  
 +\int_{\partial T} \tau_{h,n} \cdot \nabla v\, ds \right)  \\
   &\   - \sum_{E\in\cE_h} \int_E \tau_{h,nn} \left\jump{\partial_n v\right} \, ds.
\end{split} 
\end{equation}
We split  $\tau_{h,n} = \tau_{h,nt} t + \tau_{h,nn} n$ and observe by reordering the boundary integrals on $\partial T$ edge-wise, using the continuity of $\tau_{h,nn}$ and $\partial_t v$
\begin{equation}
\sum_{T \in \cT_h}\int_{\partial T} \tau_{h,n} \cdot \nabla v\, ds =
\sum_{E \in \cE_h} \int_E (\tau_{h,nn} \left\jump{\partial_n v\right} + \jump{\tau_{h,nt}}\, \partial_t v)\, ds.
\label{boundintegrals}
\end{equation}
Using \eqref{boundintegrals} in \eqref{divdiv1}, we see that the edge integrals containing $\tau_{h,nn}$ cancel. Moreover, $\partial_t v = 0$ on $\Gamma$ for $w \in H^2_0 + \Vnull$, thus we can restrict the sum to edges $E \in \cE_h^0$ in the interior of $\Omega$,
\beq
\langle \opdiv\opdiv \tau_h, v\rangle 
	= -\sum_{T \in \cT_h}\int_T \opdiv \tau_h \cdot \nabla v\, dx
	+ \sum_{E \in \cE_h^0} \int_{E}  \jump{\tau_{h,nt}}\, \partial_t v\, ds .  
\label{eq:defdivdiv_divI}
\eeq
In the next step, integration by parts is performed on each element $T$ and on each edge $E$, 
\begin{eqnarray*}
  \langle \opdiv\opdiv \tau_h, v\rangle 
  &=& \sum_{T \in \cT_h} \left(\int_T \opdiv\opdiv \tau_h \, v \, dx  - 
  \int_{\partial T} (\opdiv \tau_h) \cdot \nE \, v\,ds \right) \quad\\
  & &+ \sum_{E \in \cE_h^0} \left( -\int_{E} \jump{\partial_t\tau_{h,nt}} \, v \, ds 
 +\left(\jump{\tau_{h,nt}}\,v\right)|_{V_1(E)}^{V_2(E)} \right).
\end{eqnarray*}
Collecting element, edge, and vertex terms gives the desired representation.
\begin{subequations}
\label{ingredients}
\begin{eqnarray}
  \langle \opdiv\opdiv \tau_h, v\rangle   \label{eq:divdivmh} 
  &=& \sum_{T \in \cT_h} \int_T
  \underbrace{\opdiv\opdiv \tau_h}_{\in P^{k-3}(T)} \, v \, dx  \label{eq:divdivmT}\\
  &+ & \sum_{E \in \cE_h^0} \int_{E}
  \underbrace{\left\jump{ -\partial_t \tau_{h,nt}
  - (\opdiv \tau_h) \cdot \nE  \right}}_{\in P^{k-2}(E)} v\,ds  \label{eq:divdivmE}\\
  &+& \sum_{V \in \cV_h^0} \underbrace{\sum_{E \supset V} \delta(E,V)\jump{\tau_{h,nt}(V)}}_{\in \mathbb R} v(V). \label{eq:divdivmV}
\end{eqnarray}
\end{subequations}
In the last line $\delta(E,V)$ is a factor of $\pm 1$, evaluating to $+1$ if the vertex $V$ is the second vertex $V_2(E)$ of the oriented edge $E$, or to $-1$, if $V$ is the first vertex $V_1(E)$. One may take this sum as the jump of the jumps of the normal-tangential component of $\tau_h$ in vertex $V$ times the unique value of $v(V)$. Again, one can see that the product of $\delta(E,V)$ and $\jump{\tau_{h,nt}}$ does not depend on the orientation of $E$.

This representation fits with the degrees of freedom given in \eqref{basisduala}-\eqref{basisdualc}, therefore $\opdiv\opdiv\tau_h \in (\Vnull)^*$.
Eventually we observe that \eqref{ingredients} can be evaluated also for $v \in \Vnull \not \subset H^2_0$.
Equ. \eqref{ingredients} and \eqref{eq:defdivdiv_eps}
provide the same extension of $\langle \opdiv\opdiv \tau_h, v\rangle $ for $v \in H^2_0(\Om)+\Vnull$.
\end{proof}

\section{Data oscillation} \label{sec:dataosc}

Since the numerical solution of the equilibration condition
$\opdiv\opdiv\si_h^{eq}=f$ belongs to a finite dimensional space,
we obtain only an exact solution for a modified right-hand side $ f_h$.
Usually  this function is an $L_2$ projection of $f$ onto piecewise polynomial
functions of lower degree, see, e.g., \cite{AIR10, BFH14, BHL15}.
A similar effect is well known for residual a posteriori error estimates;
c.f., \cite{BGS10, GHV11} or \cite[p.60]{VER13},
where it is known as {\em data oscillation} for a long time.
Here, \eqref{def-f_h} shows that the discretization yields
a projection onto $(\Vnull)^*$, 
and a duality technique will be useful.


We apply \eqref{form} to $\tau_h:=\si_h^{eq}$.
It follows from Lemma~\ref{lem:divdivmwh} and the definition \eqref{I_h} of the interpolation operator $I_h$
that for $w \in H^2_0(\Omega)$
\begin{eqnarray}
\label{f_h-I_h}
    \langle f_h, I_h w\rangle &=& \sum_{V \in \cV_h^0} f^{(V)}_{\sigma_h^{eq}} I_h w(V)
 + \sum_{E \in \cE_h^0} \int_E f^{(E)}_{\sigma_h^{eq}} I_h w\, ds
 +  \sum_{T \in \cT_h} \int_Tf^{(T)}_{\sigma_h^{eq}} I_h w\, dx \nn\\ 
 &=& \sum_{V \in \cV_h^0} f^{(V)}_{\sigma_h^{eq}} w(V)
 + \sum_{E \in \cE_h^0} \int_E f^{(E)}_{\sigma_h^{eq}} w\, ds
 +  \sum_{T \in \cT_h} \int_Tf^{(T)}_{\sigma_h^{eq}} w\, dx \\
 &=&  \langle f_h,  w\rangle .\nn
\end{eqnarray}
Moreover, let $\bar f$ denote the $L_2$ projection of $f$ onto the (discontinuous) space of piecewise polynomials of degree $k-3$ in $\cT_h$,
i.e., two different projections are involved. In the lowest order case of $k = 2$, we set $\bar f = 0$.
Since $ \bar f  \in (\Vnull)^*$, similarly as in \eqref{f_h-I_h}
we see that
\beq
\label{f_k-3-Ih}
  (\bar f , I_h w)_0 = (\bar f ,  w)_0
  \quad\text{for all } w \in H^2_0(\Om).
\eeq

Let $\hat u$ denote the solution of the biharmonic equation with the
modified right hand side $f_h \in H^{-2}$,
\begin{equation}
    \int_{\Om} \nabla^2 \hat u : \nabla^2 v \,dx = \left< f_h,v\right>
  \quad\text{ for all~} v\in H^2_0(\Om). \label{eq:hatu}
\end{equation}
Then, $\sigma_h^{eq}$ is an equilibrated tensor for the solution $\hat u$.
For completing the analysis we 
estimate the  error $\eta^{osc} = \|\nabla^2(u-\hat u)\|_{0,\Om}$ that arises from
the data oscillation.

\begin{lemma}
\label{lem:data-osc}
Let $f \in L_2(\Om)$, $\bar f$ the element-wise $L^2$ projection of $f$ as above, and set $f_h = \opdiv\opdiv\si_h^ {eq}$. Let $u \in H^2_0$ denote the solution to the biharmonic problem \eqref{weak}, and $\hat u$ be the solution to the modified problem \eqref{eq:hatu}. Then the difference between $u$ and $\hat u$ is bounded by
\begin{equation} 
\label{u-hatu}
  \eta^{osc} = |u - \hat u|_2 = \|\opdiv\opdiv\sigma_h^{eq}-f\|_{-2}
   \leq c \left(\sum_{T \in \cT_h} h_T^{4}\|f- \bar f \|_{0,T}^2\right)^{1/2}. 
\end{equation} 
\end{lemma}

\begin{proof}
We rename the error $z := u - \hat u$
and observe that, by the definition of $\hat u$,
\begin{eqnarray}
  \int_{\Om} \nabla^2 z : \nabla^2 v \,dx
  &=&  \langle f- f_h, v\rangle \label{eq:z1}
    = (f, v)_{0,\Om} - \langle f_h, v \rangle.
\end{eqnarray}
{From}  Lemma \ref{lem:divdivmwh} and \eqref{divdiv-f-2}  it follows that
$(f,I_h z)_{0,\Om} = \langle \opdiv\opdiv\si_h^{eq}, I_h z \rangle
 = \langle f_h, I_h z \rangle$.
Combining this fact with \eqref{f_k-3-Ih},  choosing $v = z$ in \eqref{eq:z1}
we arrive at
\begin{equation}
\|\nabla^2 z\|_{0,\Omega}^2 = (f, z-I_h z)_{0,\Om} - \langle f_h , z-I_h z \rangle - 
  ( \bar f  , z-I_h z )_{0,\Omega}. \label{nabla2z}
\end{equation}
The second term on the right hand side of \eqref{nabla2z} vanishes due to \eqref{f_h-I_h}.
Recalling the approximation property \eqref{eq:interrorl2} of the interpolation operator $I_h$ we get
\begin{eqnarray}
\label{6.6}
  \|\nabla^2 z\|_{0,\Omega}^2
  &=&    \sum_{T \in \cT_h} (f-\bar f , z-I_h z)_{0,T}\nn\\
  &\le&  \sum_{T \in \cT_h} \|f-\bar f \|_{0,T} \;ch_T^{2} \|\nabla^2 z\|_{0,T} \nn\\
	&\le& c\left(\sum_{T \in \cT_h} h_T^{4}\|f-\bar f \|_{0,T}^2 \;\right)^{1/2} \|\nabla^2 z\|_{0,\Om}.
\end{eqnarray}
A division by $\|\nabla^2 z\|_{0,\Omega}$ yields \eqref{u-hatu},
and the proof is complete
\end{proof}


Lemma \ref{lem:data-osc} and \eqref{eq:errorest_imp} yield the area-based terms of the
final error estimate in the DG-norm \eqref{IPDGNorm1}.
The jumps of $\partial_n u_h$ across element edges are added in a further contribution $\eta^{jump}$. Theorem~\ref{theo:final} below summarizes these results.

\begin{theorem} \label{theo:final}
The error $\|u - u_h\|_{DG}$ measured in the mesh-dependent $DG$ norm is bounded by the terms
\begin{equation}
\|u-u_h\|_{DG} \leq \left((\eta^{mean})^2 + (\eta^{jump})^2\right)^{1/2} + \frac12 \eta^{eq} + \eta^{osc} \label{6.7}
\end{equation}
where from the additive parts given below only the contribution of the 
data oscillation $\eta^{osc}$ contains a generic constant,
\begin{subequations}
\begin{eqnarray} 
\label{etaa}
\eta^{mean} &=& \|\nabla^2 u_h - \sigma^{mean} \|_{0,\Om},\\
\eta^{jump} &=& \left(\sum\limits_{E \in \cE_h(\bar{\Om})}
         \frac{\alpha}{h_E} \ \| \, \jump{\dn u_h} \|^2_{0,E} \right)^{1/2},\\
\eta^{eq} &=&  \|\nabla^2 u^{conf} - \sigma^{eq}\|_{0,\Om},\\
\label{etad}
\eta^{osc} &=& c \left( \sum_{T \in \cT_h} h_T^4 \|f-\bar f\|_{0,T}^2\right)^{1/2}.
\end{eqnarray}
\end{subequations}
\end{theorem}

\begin{proof}
The proof is almost complete from \eqref{eq:errorest_imp}, we only need to treat the jump terms in the DG norm,
\begin{eqnarray*}
\|u - u_h\|_{DG}
 &=& \left( |u - u_h|_{2,h}^2 + \sum_{E \in \cE_h(\bar\Om)} \frac{\alpha}{h} \|\jump{\partial_n u_h}\|_{0,E}^2 \right)^{1/2} \nn\\
&\leq& \left( \|\sigma^{mean} - \nabla_h^2 u_h\|_{0,\Omega}^2 + \sum_{E \in \cE_h(\bar\Om)} \frac{\alpha}{h} \|\jump{\partial_n u_h}\|_{0,E}^2 \right)^{1/2} + \nn \\
&&\|\nabla^2 \hat u - \sigma^{mean}\|_{0,\Omega} + \|\nabla^2 u - \nabla^2 \hat u\|_{0,\Omega}.
\end{eqnarray*}
By inserting the definitions \eqref{etaa}--\eqref{etad} we complete the proof.
\end{proof}


\begin{remark}  \label{rem:dataosc}
The constant $c$ in {Lemma~\ref{lem:data-osc}} and Theorem~\ref{theo:final}
can be bounded by
\beq 
\label{6.8}
  c \le 0.3682146
\eeq
\cite{CC16} due to an estimate of the interpolation by the Morley element
\cite{CG14}.
\end{remark}

We will use this explicit bound in Section \ref{NumRes}.

\section{Efficiency}
\label{sec:efficiency}

The efficiency of the new error bound will follow from a comparison
with a residual error estimator that is known to be efficient \cite{BGS10,FHP15,GHV11}.
When used as an upper bound, the new error bound contains no generic constant.
A lower bound, however, is derived only with an unknown generic constant.

\begin{lemma}
\label{lem:7.1}
If $T\in \cT_h$ and $\tau_h \in [P^{k-1}(T)]^{2\times2}_{sym}$, then
\begin{eqnarray}
  \|\tau_h\|_{0,T}^2 &\le& ch \|\tau_{h,nn}\|_{0,\dT}^2 +\nn\\
  &&c \max \left\{ \int_T \tau_h:q\,dx; ~q \in  [P^{k-2}(T)]^{2\times2}_{sym},
  \int_T q:q\,dx\le 1 \right\}^2, ~
\end{eqnarray}
with a constant $c$ which depends only on $k$ and the shape parameter
of $\cT_h$.
\end{lemma}

Since the space $[P^{k-1}(T)]^{2\times2}_{sym}$ is finite dimensional,
the inequality follows from Lemma \ref{lem1} by a standard scaling argument.

To show efficiency, we establish a bound of the equilibrated error estimate $\|\sigma_h^{eq} - \nabla^2 u_h\|_{0,T}$ on each element $T$ from above. The choice of $\sigma_h^{eq}$ in \eqref{constr-eq-DG} yields 
\begin{eqnarray}
\label{8.3}
  \int_T (\si_h^{eq}-\nabla^2 u_h):q \, dx
  &=& \int_{\dT} \gamma_E\jump{\partial_n u_h}\, q_{nn}\,ds  \nn\\
  &\le& \|\,\jump{\partial_n u_h}\|_{0,\dT}\|q_{nn}\|_{0,\dT} \nn\\[3pt]
    &\le& h^{-1/2} \|\,\jump{\partial_n u_h} \|_{0,\dT} \, \|q\|_{0,T}
\end{eqnarray}
by a scaling argument for $q \in [P^{k-2}(T)]^{2\times2}_{sym}$. Similarly,  on each edge $E \subset \partial T$
\[
  \si_{h,nn}^{eq}- \mean{\nabla^2 u_{h,nn}}
  =  \frac{\alpha}{h} \jump{\partial_n u_h}  \,.
\]
Algebraic manipulation allows to express the one-sided value $(\nabla^2 u_h)_{nn}|_{\partial T}$ in terms of jumps and averages on interior edges $E \subset \dT$, 
\begin{eqnarray*}
 \si_{h,nn}^{eq}-\nabla^2 u_{h,nn}|_{\dT}
 &=&\si_{h,nn}^{eq}-\mean{(\nabla^2 u_h)_{nn}}
 \pm\frac12 \jump{(\nabla^2 u_h)_{nn}} \\
  &=& \frac{\alpha}{h} \jump{\partial_n u_h} \pm \frac12 \jump{(\nabla^2 u_h)_{nn}} \,.
\end{eqnarray*}
Here, the sign of the second jump term in the last line depends on the orientation of the edge, namely it is negative if $T = T_1(E)$ and positive if $T = T_2(E)$. However,
we will refer only to the absolute value,
and the next inequality holds for both cases, and also for boundary edges,
\beq
\label{8.4}
  \| \si_{h,nn}^{eq}-(\nabla^2 u_h)_{nn}|_{\dT}\|_{0,E}
  \le \frac{\alpha}{h} \|\,\jump{\partial_n u_h}\|_{0,E}
  + \frac12 \|\, \jump{(\nabla^2 u_h)_{nn}}\|_{0,E} \,.
\eeq
We apply Lemma \ref{lem:7.1} to $\tau_h = \si_h^{eq}-\nabla^2 u_h$,
collect the terms in \eqref{8.3} and \eqref{8.4}, and recall Young's inequality,
\[
  \|\si_h^{eq}-\nabla^2 u_h\|_{0,T}^2
   \le c\sum_{E\in\dT} \Big(
   h^{-1} (1+\alpha)^2\|\,\jump{\partial_n u_h}\|_{0,E}^2
   +h\|\, \jump{(\nabla^2 u_h)_{nn}}\|_{0,E}^2 \, \Big).
\]
The terms on the right-hand side belong to the well-known
residual a posteriori error estimates in \cite{BGS10, FHP15,GHV11}.

The additional term $\|u_h-u^{conf}\|_{DG}$ is known to do not spoil
the efficiency. Eventually, the data oscillation is a term of higher order.
The a posteriori
error bound \eqref{eq:errorest}, and a fortiori the 
improved bound from Theorem~\ref{theo:final} is efficient.

The comparison between the two different methods is not only a global one,
but also local. Therefore, the new error bound is expected to be suitable also for
local refinement techniques.

\section {Equilibration for the Hellan--Herrmann-- \break Johnson method}

We will see that an equilibration
 for the Hellan--Herrmann--Johnson method  \cite{COM80} can be obtained
in a few lines, since the finite element spaces
$\Vnull$ and $M_h$ are the same as above.

To this end we rewrite the mixed formulation in \cite{COM80} with our symbols:
{\em Find $\sigma^{HHJ}_h \in M_h$ and $u_h \in \Vnull$}
such that
\beq
\label{HHJ-com}
\begin{array}{llll}
    a(\sigma^{HHJ}_h, \tau_h) + b(\tau_h,u_h) &=& 0 
    &\text{for all }\tau_h\in M_h\,, \\[3pt]
      b(\sigma^{HHJ}_h,v_h)                    &=& \ds-\int_{\Om} f v_h\,dx
    &\text{for all } v_h\in \Vnull\,.
\end{array}
\eeq
where
\begin{subequations}
\begin{eqnarray}
\label{bsiva}
  a(\sigma_h, \tau_h) &:=&  \int_{\Om} \sigma_h :\tau_h \,dx,  \\
\label{bsiv}
  b(\tau_h,v_h)   &:=&   \sum_T \Big(
   \int_T\opdiv \tau_h \cdot \nabla v_h \,dx -\int_{\dT} \tau_{h,nt}\, \partial_t v_h\, ds \Big), 
\end{eqnarray}
\end{subequations}
Note that we have changed a sign on the right-hand side of \eqref{HHJ-com}
in order to be consistent with \eqref{HHJ}.
Reordering the boundary terms in \eqref{bsiv} leads to the negative of the right-hand side of formula \eqref{eq:defdivdiv_divI}, i.e.,
\[
  b(\tau_h,v_h) = -\langle \opdiv\opdiv \tau_h, v_h\rangle.
\]
Thus the second line of \eqref{HHJ-com} ensures
\[
  \left< \opdiv\opdiv \sigma_h^{HHJ},v_h \right> = -b(\sigma_h^{HHJ},v_h)
  = \int_{\Om} fv_h\, dx
  \quad \text{for all } v_h \in V_h \,.
\]  
Similarly as with \eqref{divdiv-f-2} we conclude 
that $\sigma^{eq}_h:=\si^{HHJ}_h$ satisfies the relation \eqref{equilibratedtau}
of the first step in the equilibration procedure.
Thus the mixed method due to Hellan--Herrmann--Johnson 
provides an equilibrated moment tensor for the first aim immediately.
A common treatment with the discontinuous Galerkin method
is natural for the remainder
of the analysis. For this reason we refer to the analogous  considerations
in the previous sections.

The mixed method by Hellan--Herrmann--Johnson  is considered as nonconforming,
since the operator $\opdiv\opdiv$ does not send the tensor-valued functions
in $M_h$ to $L_2(\Om)$. Therefore the functions in $M_h$ are not
candidates for equilibrated tensors in an elementary manner.
If the operator is understood in the distributional sense,
there is no problem with the maximum problem \eqref{minprob}
nor with Theorem \ref{theo-2en}.
The concept of Hellan--Herrmann--Johnson looks very natural in this framework.
If it is considered as nonconforming,
then it is nonconforming only in a weak way.

\section {Numerical results}
\label{NumRes}

We present our results for the performance of the error estimator for two examples with known analytical solution.
 In the implementation, we used a hybrid DG formulation, where the jump $\jump{\partial_n u_h}$ is discretized by an extra unknown of order $k-1$ on element edges.



\subsection{Example 1: Solution with singularity}

The example from \cite{Gri92}, which is found also in \cite{BHL15}, contains
the L-shaped domain $\Omega := (-1, 1)^2 \backslash( [0,1)\times(-1,0])$ with angle $\omega = 3\pi/2$ at the re-entrant corner.
The right hand side $f \in L^2(\Omega)$ is chosen such that the singular solution $u \in H^2_0(\Omega)$
is given in polar coordinates by
\begin{equation}
\label{sol9.1}
   u(r, \phi) = \left(r^2 \cos^2(\phi) - 1\right)^2\left(r^2 \sin^2(\phi) - 1\right)^2 r^{1+z} g(\phi),
\end{equation}
where  $z = 0.5444837$ is
a non-characteristic root of $\sin^2(\omega z) = z^2 \sin^2(\omega)$ and
\begin{equation}
\begin{split}
&g(\phi) =\\
& \left(\tfrac{1}{z-1} \sin((z-1)\omega) - \tfrac{1}{z+1} \sin((z+1)\omega)\right)
\left( \cos((z-1)\phi) - \cos((z+1)\phi) \right)\\
&  -\left(\tfrac{1}{z-1} \sin((z-1)\phi) - \tfrac{1}{z+1} \sin((z+1)\phi)\right)
\left( \cos((z-1)\omega) - \cos((z+1)\omega) \right).
\end{split}
\end{equation}
The penalty parameter in the DG formulation \eqref{IPDGBilForm1} is set to $\alpha = (k+1)^2$.

{%
Computations were done with the DG finite element spaces $\Vnull$ for the orders $k=2$ and $k=3$.
The mesh was refined adaptively, where elements $T$ satisfying the relative criterion
\begin{equation}
\eta^{eq}(T) > 0.25  \max(\eta^{eq})
\end{equation} 
were marked for refinement.
A conforming approximation $u^{conf}$ was determined 
for the lowest-order case $k=2$ by an $L^2$ projection to the rHCT space of
reduced Hsieh--Clough--Tocher elements \cite{Cia78},
and by the projection to the full Clough--Tocher space \cite{CT65} for the case $k=3$, respectively.
The space of the equilibrated moment tensors $M_h$ is of order $k-1$ in both cases.
The contributions to the basic and improved error estimates \eqref{eq:errorest} and \eqref{eq:errorest_imp} are depicted in figures~\ref{fig:conv_p2} and \ref{fig:conv_p3}.
Results are also displayed in Table~\ref{tab:conv_p2}.

We compute the efficiency of the error estimate according to \eqref{eq:errorest} including the additional jump terms of the DG norm as
\begin{equation}
\text{eff}^{eq} = \frac{((\eta^{nonconf})^2 + (\eta^{jump})^2)^{1/2} + \eta^{eq} + \eta^{osc}}{\|u - u_h\|_{DG}}
\end{equation}
and the corresponding numbers for the improved error estimate due to Theorem~\ref{theo:final}
\begin{equation}
\text{eff} = \frac{((\eta^{mean})^2+(\eta^{jump})^2))^{1/2}+ \tfrac12\eta^{eq} + \eta^{osc}}{\|u - u_h\|_{DG}}.
\end{equation}
We find   $\text{eff} = 1.45$ for $k=2$ and $\text{eff} = 1.88$ for $k=3$ on the finest mesh;
see also the results in Table~\ref{tab:conv_p2}.
}

In both cases, the term $\eta^{eq}$ due to equilibration  is dominating. This leads to an increase of efficiency in the improved error estimate, where this contribution is cut by half. 

For the lowest-order case $k=2$, the error due to rHCT interpolation $\eta^{nonconf}$ is visibly smaller than the error contribution $\eta^{mean}$ due to the difference to averaged moment tensor $\sigma^{mean}$. However, for $k=3$, these estimates are much closer, and also very close to the exact error $\|u - u_h\|_{DG}$.

The data oscillation $\eta^{osc}$ is estimated as described in Lemma~\ref{lem:data-osc} with the factor from Remark~\ref{rem:dataosc}. We see that $\eta^{osc}$ is very high for very coarse discretizations. However, it is of higher order than all other contributions, and becomes negligible for realistic discretizations. 

{%
We note that the contribution $\eta^{nonconf}$ of the nonconformity is smaller than the
contribution of the jump terms, and both ones are small for fine grids.
Since the computation of  $\eta^{nonconf}$ requires $H^2$ elements, which one wants to avoid by the DG method in the first place,
it may be justified to neglect it in the computation of the a posteriori error bound.
}


\begin{figure}
\begin{center}
\includegraphics[width = 0.7\textwidth]{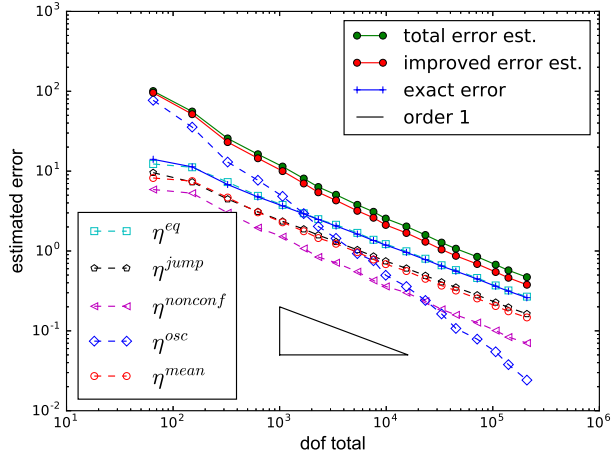}
\end{center}
\caption{Example 1: convergence of the error components for polynomial order $k = 2$, adaptive refinement based on $\eta^{eq}$.}
\label{fig:conv_p2}
\end{figure}

\begin{figure}
\begin{center}
\includegraphics[width = 0.7\textwidth]{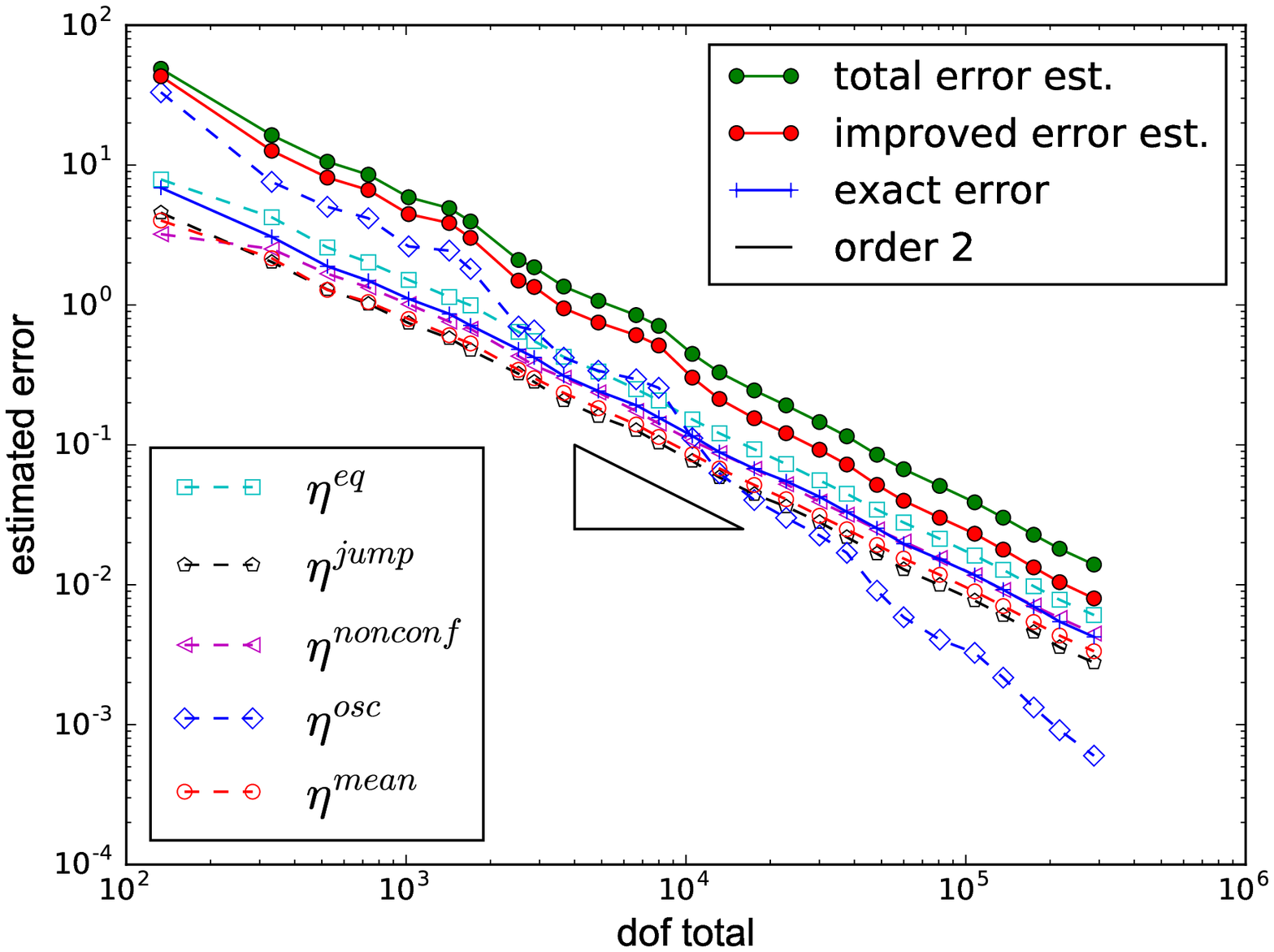}
\end{center}
\caption{Example 1: convergence of the error components for polynomial order $k = 3$, adaptive refinement based on $\eta^{eq}$.}
\label{fig:conv_p3}
\end{figure}

\begin{table}[ht]
\caption{Numerical results showing the 
size of the
contributions to the error bound in Example~1 for polynomial order $k = 2$.} \label{tab:conv_p2}
\vspace{-6pt}
\begin{center}
\footnotesize
\begin{tabular}
[c]{|r|l|lllll|cc|}%
\hline
dof V& exact err   & $\eta^{eq}$&$\eta^{nonconf}$ & $\eta^{osc}$& $\eta^{mean}$& $\eta^{jump}$& $\text{eff}^{eq}$ & $\text{eff} $ \vrule height 12pt depth 5pt width 0pt\\
\hline
65 & 14.05 & 12.28 &	 5.88&  	 77.04&  	 8.20 & 	 9.56&   7.15& 6.81\vrule height 12pt width 0pt\\
625 	& 4.75& 4.90&  	 1.95& 	 7.68& 	 3.04&  3.10& 	 3.42&  3.05\\
5357 &	 1.63& 	 1.68&	 0.54&	 0.92& 	 0.96& 1.03&	 2.32 & 	 1.95\\
45059 &	 0.558& 	 0.576&	 0.158&	 0.107& 0.319&	 0.350& 1.91 & 	 1.55 \\
106386 &	 0.361& 	 0.370& 	 0.101 & 	 0.054 & 0.204 & 	 0.227&	 1.86& 	 1.51\\
208986 &	 0.260& 	 0.268&	 0.070&	 0.024 & 0.147& 	 0.163& 1.80& 	 1.45\\
\hline
\end{tabular}
\end{center}
\end{table}

\subsection{Example 2: clamped, simply supported and free boundary} 
\label{sec:numres2}

In order to show the flexibility of the method we consider an example from \cite{TW59}.
The plate covers the unit square $\Omega = (0,1)^2$, there is a 
uniform load $f=1$, and clamped, simply supported and free boundaries occur.
The plate is
\begin{subequations}
\label{genbc}
\begin{align}
  &\text{simply supported,} \ u =\ 0, \ (\nabla^2 u)_{nn} =\ 0, && \text{for } x = 0 \text{ and } x=1,\\
  &\text{clamped,}\ u =\ 0, \ \partial_n u =\ 0,&& \text{for } y=0,\\
  &\text{free,}\  (\nabla^2 u)_{nn} =\ 0, \ K_n(\nabla^2 u)\cdot n =\ 0, && \text{for } y=0.
\end{align}
\end{subequations}
On the free boundary, $ K_n(\nabla^2 u) :=   \opdiv(\nabla^2 u) \cdot n+ \partial_t (\nabla^2 u)_{nt}$ is the boundary shear force.
The associated boundary parts are denoted as
$\Gamma_{S}$, $\Gamma_{C}$, and $\Gamma_{F}$, respectively.

\begin{remark} \label{rem:bc} %
In an $H^2$ conforming finite element method for the biharmonic equation, the essential boundary conditions are those on $u_h$ and $\partial_n u_h$. Conditions on $(\nabla^2u_h)_{nn}$ and $K_n(\nabla^2 u_h)$ are natural and, if inhomogeneous, enter into the right hand side of the variational equation \eqref{weak}. These conditions are then satisfied in weak sense only.

This is fundamentally different in the mixed Hellan--Herrmann--Johnson method and also the equilibration process.
Here, the essential conditions are those on $u_h$ and $\sigma_{h,nn}$.
Conditions on $\partial_n u_h$ and $K_n(\sigma_h)$ are natural
and satisfied in weak sense. 
We will elucidate the treatment of the different boundary conditions \eqref{genbc}
in the subsequent remark. 
\end{remark}

\begin{remark}
\label{obc}
The variational formulation \eqref{weak} refers to $\partial \Om=\Gamma_{C}$.
Now we deal with the adaptation for the boundary conditions \eqref{genbc}.
First, the condition ``for all $w\in H^2_0(\Om)$'' 
has to be replaced by
\begin{align}
 \text{for all }&w \in \hat H^2(\Om) := \{w\in H^2(\Om) : w = 0 \text{ on } \Gamma_C \cup \Gamma_S, \partial_n w = 0 \text{ on  }\Gamma_C\}. \nn
\end{align}
Obviously this applies to many equations.
In particular, the distributional definition \eqref{divdivdistr} is still valid.
Here we assume that $\sigma \in [L_2(\Om)]^{2\times2}_{sym}$ is sufficiently smooth such that the  boundary condition $\sigma_{nn} = 0$ is well defined on the free and simply supported boundary parts $\Gamma_F$ and $\Gamma_S$. 
Also the finite element functions have to satisfy the homogeneous essential boundary conditions $v_h=0$ and $\tau_{h,nn} = 0$ on their respective boundary parts. Then the extension of the double divergence operator to the finite element space \eqref{eq:defdivdiv_eps}, and its element-wise representations \eqref{eq:defdivdiv_divI} and \eqref{ingredients} are still valid. The edges and vertices on $\Gamma_F$ are included in $\cE_h^0$ and $\cV_h^0$, respectively. 

In the DG scheme, the different boundary conditions are realized as follows: 
\begin{itemize}
\item
The boundary condition  $u=0$ on $\Ga_C\cup\Ga_S$ is essential and enforced
by considering  in the  the variational formulation only the functions
in $V_h$ with this property.
Otherwise the natural boundary condition
$  K_n(\nabla^2 u) = 0$ on $\Gamma_F$ is
achieved in weak sense by the adapted variational formulation \eqref{augvar}.
\item
The boundary condition  $\dn u=0$  on $\Ga_{C}$
is essential and enforced approximately by the
penalty terms on $\Ga_{C}$.
There are no edge penalty terms on $\partial\Om \backslash \Gamma_{C}$ in the adapted variational formulation \eqref{augvar}, which implies the natural boundary condition $(\nabla^2 u)_{nn}=0$ on 
$\partial\Om\backslash\Ga_{C} $.
\end{itemize}
The adapted DG formulation reads
\begin{equation}
\label{augvar}
  A_h(u_h,v_h) = (f,v_h)_0  
\qquad \text{for all ~} v_h \in V_h \text{~ with } v_h(x)=0, ~x\in \Ga_C \cup \Gamma_S. 
\end{equation}
Here we understand $A_h$ as in \eqref{IPDGBilForm1} after the edge integrals on $\Ga_S\cup\Ga_F$
have been canceled.

In the equilibration process, we respect the essential boundary condition $\sigma^{eq}_{h,nn} = 0$ on $\Gamma_S \cup \Gamma_F$.
%
The construction rule \eqref{constr-eq-DG} for $\sigma^{eq}_{h}$ on an element $T\in\cT$
is now generalized
\[
\begin{split}
   \si^{eq}_{h,nn} & =
     \begin{cases}  0 & \text{on } \partial T \cap (\Ga_{S}\cup \Gamma_F)  , \\
      \text{as in \eqref{constr-eq-DG}} & \text{otherwise,} 
     \end{cases}
                 \\[6pt]
  \ds \int_T \si^{eq}_h : q_T\,dx & =  \ds \int_T \nabla^2 u_h : q_T \,dx
  - \sum_{E \subset \partial T\backslash (\Ga_S\cup\Ga_F)} \int_{E} \gamma_E\jump{\partial_n u_h}\, q_{T,nn}\,ds \\ 
	& \hspace{6cm} \forall q_T \in [P^{k-2}(T)]^{2\times 2}_{sym}.
	\end{split}
\]
{A tedious calculation shows $\langle \opdiv\opdiv \sigma^{eq}_h, v_h\rangle = A_h(u_h, v_h) = (f,v_h)_0$. }
\end{remark}

\begin{figure}
\begin{center}
\includegraphics[width = 0.7\textwidth]{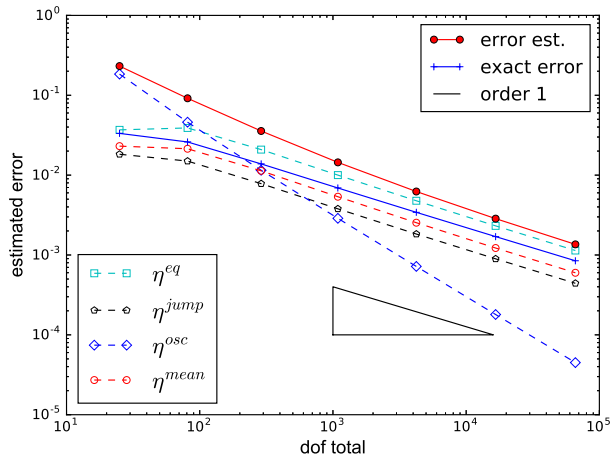}
\end{center}
\caption{Example 2: convergence of the error components for polynomial order $k = 2$, uniform refinement.}
\label{fig:BC2square_p2}
\end{figure}

The analytic solution of the example under consideration is given as a series
of trigonometric and hyperbolic functions, for details see the original work \cite{TW59}.
The domain $\Om$ is convex, and the solution is sufficiently regular to render adaptive refinement unnecessary. 

In Figure~\ref{fig:BC2square_p2} we show the convergence for a constant penalty parameter $\alpha = 2(k+1)^2$ and polynomial order $k=2$. The efficiency of the error estimate according to Theorem~\ref{theo:final} is 1.60 on the finest mesh. 



Additionally, we plot the behavior of the exact error and the error estimate components for different penalty parameters $\alpha = \alpha_0 (k+1)^2$ with $\alpha_0 \in [0.25, 8]$. Figure~\ref{fig:BC2alphap2} show the results on a mesh with 32768 elements and polynomial order $k=2$, respectively. We see that the total error stagnates for $\alpha_0 \geq 1$. While the nonconforming error estimate component $\eta^{jump}$ decrease with growing penalty parameter, the estimates based on equilibration $\eta^{eq}$ and $\eta^{mean}$ increase. The data oscillation is of course independent of the penalty parameter $\alpha_0$. The efficiency of the error estimator is best for moderate values of $\alpha_0 \simeq 1$, and increases up to $\text{eff} \simeq 2$ for large $\alpha_0$.

\begin{figure}
\begin{center}
\includegraphics[width = 0.7\textwidth]{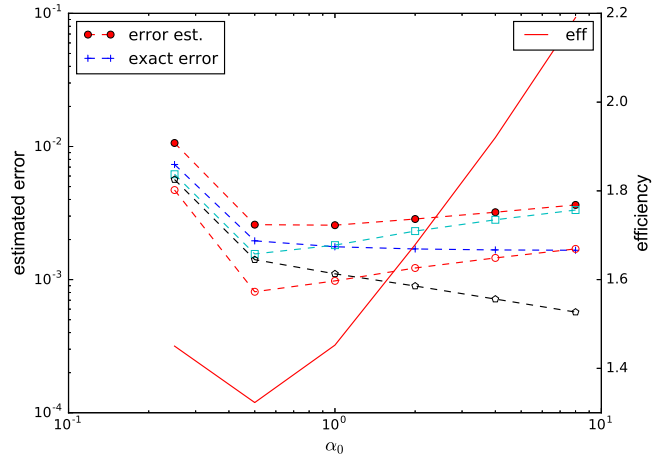}
\end{center}
\caption{Example 2: different values of $\alpha_0$, polynomial order $k = 2$. Error components $\eta^{eq}$, $\eta^{jump}$ and $\eta^{mean}$ are labelled as in Figure~\ref{fig:BC2square_p2} .}
\label{fig:BC2alphap2}
\end{figure}



\end{document}